\documentclass{amsart}

\usepackage{amssymb}
\usepackage{amsmath}
\usepackage{amsthm}
\usepackage{graphicx}
\usepackage{bm}

\newtheorem{thm}{Theorem}[section]

\theoremstyle{definition}

\theoremstyle{remark}

\numberwithin{equation}{section}

\newcommand{\CC}{\mathbb{C}}
\newcommand{\ZZ}{\mathbb{Z}}
\newcommand{\E}{\mathbb{E}}
\newcommand{\PPP}{\mathbb{P}}

\newcommand{\X}{\mathbf{X}}
\newcommand{\Y}{\mathbf{Y}}

\newcommand{\Z}{\mathbf{Z}}

\newcommand{\0}{\mathbf{0}}

\newcommand{\uu}{\mathbf{u}}

\newcommand{\UU}{\bm{U}}
\newcommand{\VV}{\bm{V}}

\newcommand{\DD}{\bm{D}}

\newcommand{\A}{\bm{A}}

\newcommand{\C}{\bm{C}}
\newcommand{\I}{\bm{I}}

\newcommand{\OO}{\bm{O}}
\newcommand{\f}{\bm{f}}

\newcommand{\etav}{\bm{\eta}}

\newcommand{\xiv}{\bm{\xi}}
\newcommand{\Phiv}{\bm{\Phi}}
\newcommand{\Siga}{\bm{\Sigma}}

\newcommand{\Lam}{\bm{\Lambda}}
\newcommand{\av}{\bm{a}}
\newcommand{\bv}{\bm{b}}

\newcommand{\phiv}{\bm{\phi}}

\newcommand{\muv}{\bm{\mu}}

\newcommand{\F}{\bm{F}}

\newcommand{\Lada}{\bm{\Lambda}}

\newcommand{\SSS}{\bm{S}}

\newcommand{\psiv}{\bm{\psi}}
\newcommand{\gammav}{\bm{\gamma}}
\newcommand{\Gammav}{\bm{\Gamma}}

\newcommand{\proj}{\mathrm{Proj}}
\newcommand{\diag}{\mathrm{diag}}

\newcommand{\cov}{\mathrm{Cov}}

\newcommand{\tr}{\, \mathrm{tr}}

\newcommand{\spanot}{\overline{\mathrm{span}}}

\begin{document}

\title[Factorization of a spectral density with smooth eigenvalues]{Factorization of a spectral density with smooth eigenvalues of a multidimensional stationary time series}

\author{Tam\'as Szabados}

\address{Budapest University of Technology and Economics}

\email{szabados@math.bme.hu}


\begin{abstract}
The aim of this paper to give a multidimensional version of the classical one-dimensional case of smooth spectral density. A smooth spectral density gives an explicit method to factorize the spectral density and compute the constituents of the Wold representation of a regular weakly stationary time series. These constituents are important to give the best linear predictions of the time series.
\end{abstract}

\maketitle

\section{Introduction}\label{se:intro}

Let $\X_t = (X_t^1, \dots , X_t^d)$, $t \in \ZZ$, be a $d$-dimensional weakly stationary time series, where each $X_t^j$ is a complex valued random variable on the same probability space $(\Omega, \mathcal{F}, \PPP)$. It is a second order, and in this sense, translation invariant process:
\[
\E \X_t = \muv \in \CC^d, \qquad \E((\X_{t+h}-\muv)(\X^*_t - \muv^*)) = \C(h) \in \CC^{d \times d}, \qquad \forall t,h \in \ZZ,
\]
where $\C(h)$, $h \in \ZZ$, is the non-negative definite covariance matrix function of the process. Without loss of generality, from now on it is assumed that $\muv = \0$.

Thus the considered random variables will be $d$-dimensional, square integrable, zero expectation random complex vectors whose components belong to the Hilbert space $L^2(\Omega, \mathcal{F}, \PPP)$. However, the \emph{orthogonality }of the random vectors $\X$ and $\Y$ is defined by the stronger relationship $\X \perp \Y \Leftrightarrow \cov(\X, \Y) = \E(\X \Y^*) = \OO$.

The \emph{past of $\{\X_t\}$ until time $n \in \ZZ$} is the closed linear span in $L^2(\Omega, \mathcal{F}, \PPP)$ of the past and present values of the components of the process:
\[
H^-_n := \spanot \{X^{\ell}_{\gamma} : \ell = 1, \dots , d; t \le n\}.
\]
The \emph{remote past} of $\{\X_t\}$ is $H_{-\infty} := \bigcap_{n \in \ZZ} H^-_n$. The process $\{\X_t\}$ is called \emph{regular}, if $H_{-\infty} = \{ 0 \}$ and it is called \emph{singular} if $H_{-\infty} = H(\X) := \spanot \{\X_t : t \in \ZZ\}$. Of course, there is a range of cases between these two extremes.

Singular processes are also called \emph{deterministic}, see e.g.\ Brockwell, Davis, and Fienberg (1991), because based on the past $H^-_0$, future values $\X_{1}$, $\X_{2}$, \dots, can be predicted with zero mean square error. On the other hand, regular processes are also called \emph{purely non-deterministic}, since their behavior are completely influenced by random \emph{innovations}. Consequently, knowing $H^-_0$, future values $\X_{1}$, $\X_{2}$, \dots, can be predicted only with positive mean square errors $\sigma^2_1$, $\sigma^2_2$, \dots, and $\lim_{t \to \infty} \sigma^2_t = \| \X_0 \|^2$. This shows why studying \emph{regular} time series is a primary target both in the theory and  applications. The Wold decomposition proves that any non-singular process can be decomposed into an orthogonal sum of a nonzero regular and a singular process. This also supports why it is important to study regular processes.

The Wold decomposition implies, see e.g.\ the classical references Rozanov (1967) and Wiener and Masani (1957), that $\{\X_t\}$ is regular if and only if $\{\X_t\}$ can be written as a causal infinite moving average (\emph{a Wold representation})
\begin{equation}\label{eq:causal_MA}
\X_t = \sum_{j=0}^{\infty} \bv(j) \xiv_{t-j}, \qquad t \in \ZZ, \qquad \bv(j) \in \CC^{d \times r} ,
\end{equation}
where $\{\xiv_t\}_{t \in \ZZ}$ is an $r$-dimensional $(r \le d)$ orthonormal white noise sequence: $\E \xiv_t = \0$, $\E(\xiv_s \xiv^*_t) = \delta_{s t} \I_r$, $\I_r$ is the $r \times r$ identity matrix. The white noise process $\{\xiv_t\}$ in the above Wold representation is unique up to a multiplication by a constant unitary matrix; so it is called \emph{a fundamental process} of the regular time series.

An important use of Wold representation is that \emph{the best linear $h$-step ahead prediction} $\hat{\X}_h$ can be given in terms of that. If the present time is $0$, then the orthogonal projection of $\X_h$ to the past $H^-_0(\X)$ is
\begin{equation}\label{eq:pred_h_step}
\hat{\X}_h = \sum_{j=h}^{\infty} \bv(j) \xiv_{h-j} = \sum_{k=-\infty}^0 \bv(h-k) \xiv_k .
\end{equation}

An alternative way to write Wold representation is
\begin{equation}\label{eq:Wold_innov}
\X_t = \sum_{j=0}^{\infty} \av(j) \etav_{t-j}, \qquad t \in \ZZ, \qquad \av(j) \in \CC^{d \times d} ,
\end{equation}
where $\{\etav_t\}_{t \in \ZZ}$ is the $d$-dimensional white noise process of \emph{innovations}:
\begin{equation}\label{eq:innovation}
\etav_t := \X_t - \proj_{H^-_{t-1}} \X_t,  \quad t \in \mathbb{Z} ,
\end{equation}
$\E \etav_t = \0$, $\E(\etav_s \etav^*_t) = \delta_{s t} \Siga$, $\Siga$ is a $d \times d$ non-negative definite matrix of rank $r$.

It is also classical that any weakly stationary process has a non-negative definite spectral measure matrix $d\F$ on $[-\pi, \pi]$ such that
\[
\C(h) = \int_{-\pi}^{\pi} e^{i h \omega} d\F(\omega), \qquad h \in \ZZ.
\]
Then $\{\X_t\}$ is regular, see again e.g.\ Rozanov (1967) and Wiener and Masani (1957), if and only if $d\F  = \f$, the spectral density $\f$ has a.e.\ constant rank $r$, and can be factored in a form
\begin{equation}\label{eq:f_factor}
\f(\omega) = \frac{1}{2 \pi} \phiv(\omega) \phiv^*(\omega) , \quad \phiv(\omega) = [\phi_{k \ell}(\omega)]_{d \times r} , \quad \text{for a.e.\ } \omega \in [-\pi, \pi] ,
\end{equation}
where
\begin{equation}\label{eq:bj_Fourier}
\phiv(\omega) = \sum_{j=0}^{\infty} \tilde{\bv}(j) e^{-i j \omega}, \quad \|\phiv\|^2 = \sum_{j=0}^{\infty} \|\tilde{\bv}(j)\|^2 < \infty,
\end{equation}
$\| \cdot \|$ denotes spectral norm. Here the sequence of coefficients $\{\tilde{\bv}(j)\}$ is not necessarily the same as in the Wold decomposition. Also,
\begin{equation}\label{eq:Phiv}
\phiv(\omega) = \Phiv(e^{-i \omega}), \quad \omega \in (-\pi, \pi], \quad \Phiv(z) = \sum_{j=0}^{\infty} \tilde{\bv}(j) z^j,  \quad z \in D,
\end{equation}
so the entries of the \emph{analytic matrix function} $\Phiv(z) = [\Phi_{j k}(z)]_{d \times r}$ are  analytic functions in the open unit disc $D$ and belong to the class $L^2(T)$ on the unit circle $T$, consequently, they belong to the Hardy space $H^2$. It is written as $\Phiv \in H^2_{d \times r}$ or briefly $\Phiv \in H^2$.

Recall that the Hardy space $H^p$, $0 < p \le \infty$, denotes the space of all functions $f$ analytic in $D$ whose $L^p$ norms over all circles $\{z \in \CC : |z| = r\}$, $0 < r < 1$, are bounded, see e.g.\  Rudin (2006, Definition 17.7). If $p \ge 1$, then equivalently, $H^p$ is the Banach space of all functions $f \in L^p(T)$ such that
\[
f(e^{i \omega}) = \sum_{n=0}^{\infty} a_n e^{i n \omega} , \qquad \omega \in [-\pi, \pi],
\]
so the Fourier series of $f(e^{i \omega})$ is one-sided, $a_n = 0$ when $n < 0$, see e.g.\ Fuhrmann (2014, Section II.12). Notice that in formulas \eqref{eq:bj_Fourier} and \eqref{eq:Phiv} there is a negative sign in the exponent: this is a matter of convention that I am going to use in the sequel too.

 An especially important special case of Hardy spaces is $H^2$, which is a Hilbert space, and which by Fourier transform isometrically isomorphic with the $\ell^2$ space of sequences $\{a_0, a_1, a_2, \dots\}$ with norm square
\[
\frac{1}{2 \pi} \int_{-\pi}^{\pi} |f(e^{i \omega})|^2 d\omega = \sum_{n=0}^{\infty} |a_n|^2 .
\]

For a one-dimensional time series $\{X_t\}$ ($d=1$) there exists a rather simple sufficient and necessary condition of regularity given by Kolmogorov (1941):
\begin{itemize}
  \item[(1)] $\{X_t\}$ has an absolutely continuous spectral measure with density $f$,
  \item[(2)] $\log f \in L^1$, that is, $\int_{-\pi}^{\pi} \log f(\omega) d\omega > -\infty$.
\end{itemize}
Then the Kolmogorov--Szeg\H{o} formula also holds:
\[
\sigma^2 = 2\pi \exp \int_{-\pi}^{\pi} \log f(\omega) \frac{d \omega}{2\pi}  ,
\]
where $\sigma^2$ is the variance of the innovations $\eta_t := X_t - \proj_{H^-_{t-1}} X_t$, that is, the variance of the one-step-ahead prediction.

For a multidimensional time series $\{\X_t\}$ which has \emph{full rank}, that is, when $\f$ has a.e.\ \emph{full rank}: $r=d$, and so the innovations $\etav_t$ defined by \eqref{eq:innovation} have full rank $d$, there exists a similar simple sufficient and necessary condition of regularity, see Rozanov (1967) and Wiener and Masani (1957):
\begin{itemize}
  \item[(1)] $\{\X_t\}$ has an absolutely continuous spectral measure matrix $d\F$ with density matrix $\f$,
  \item[(2)] $\log \det \f \in L^1$, that is, $\int_{-\pi}^{\pi} \log \det \f(\omega) d\omega > -\infty$.
\end{itemize}
Then the $d$-dimensional Kolmogorov--Szeg\H{o} formula also holds:
\begin{equation}\label{eq:Kolm_Szego_mult}
\det \Siga = (2\pi)^d \exp \int_{-\pi}^{\pi} \log \det \bm{f}(\omega) \frac{d \omega}{2\pi}  ,
\end{equation}
where $\Siga$ is the covariance matrix of the innovations $\etav_t$ defined by \eqref{eq:innovation}.

On the other hand, the generic case of regular time series is more complicated. To my best knowledge, the result one has is Rozanov (1967, Theorem 8.1):
A $d$-dimensional stationary time series $\{\X_t\}$ is regular and of rank $r$, $1 \le r \le d$, if and only if each of the following conditions holds:
\begin{itemize}
  \item[(1)] It has an absolutely continuous spectral measure matrix $d\F$ with density matrix $\f(\omega)$ which has rank $r$ for a.e.\ $\omega \in[-\pi, \pi]$.
  \item[(2)] The density matrix $\f(\omega)$ has a principal minor $M(\omega) = \det [f_{i_p j_q}(\omega)]_{p,q=1}^r$, which is nonzero a.e.\ and
      \[
      \int_{-\pi}^{\pi} \log M(\omega) d\omega > - \infty.
      \]
   \item[(3)] Let $M_{k \ell}(\omega)$ denote the determinant obtained from $M(\omega)$ by replacing its $\ell$th row by the row $[f_{k i_p}(\omega)]_{p=1}^r$. Then the functions $\gamma_{k \ell}(e^{-i \omega}) = M_{k \ell}(\omega)/M(\omega)$ are boundary values of functions of the Nevanlinna class $N$.
\end{itemize}
It is immediately remarked in the cited reference that ``unfortunately, there is no general method determining, from the boundary value $\gamma(e^{-i \omega})$ of a function $\gamma(z)$, whether it belongs to the class $N$.''

Recall that the Nevanlinna class $N$ consists of all functions $f$ analytic in the open unit ball $D$ that can be written as a ratio $f = f_1/f_2$, $f_1 \in H^p$, $f_2 \in H^q$, $p, q > 0$, where $H^p$ and $H^q$ denote Hardy spaces, see e.g.\ Nikolski (1991, Definition 3.3.1).

The aim of this paper is to extend from the one-dimensional case to the multidimensional case
a well-known sufficient condition for the regularity and a method of finding an $H^2$ spectral factor in the case of smooth spectral density.

\section{Generic regular processes}\label{sse:generic_regular}

To find an $H^2$ spectral factor if possible, a simple idea is to use a spectral decomposition of the spectral density matrix. (Take care that here we use the word `spectral' in two different meanings. On one hand, we use the spectral density of a time series in terms of a Fourier spectrum, on the other hand we take the spectral decomposition of a matrix in terms of eigenvalues and eigenvectors.)

So let $\{\X_t\}$ be a $d$-dimensional stationary time series and assume that its spectral measure matrix $d\F$ is absolutely continuous with density matrix $\f(\omega)$ which has rank $r$, $1 \le r \le d$, for a.e.\ $\omega \in [-\pi, \pi]$. Take the parsimonious spectral decomposition of the self-adjoint, non-negative definite matrix $\f(\omega)$:
\begin{equation}\label{eq:spectr_decomp_f_short1}
\f(\omega) = \sum_{j=1}^{r} \lambda_j(\omega) \uu_j(\omega) \uu^*_j(\omega) = \tilde{\UU}(\omega) \Lada_r(\omega) \tilde{\UU}^*(\omega) ,
\end{equation}
where
\begin{equation}\label{eq:Lam_r}
\Lada_r(\omega) = \diag[\lambda_1(\omega), \dots , \lambda_r(\omega)], \quad \lambda_1(\omega) \ge \cdots \ge \lambda_r(\omega) > 0,
\end{equation}
for a.e.\ $\omega \in [-\pi, \pi]$, is the diagonal matrix of nonzero eigenvalues of $\f(\omega)$ and
\[
\tilde{\UU}(\omega) = [\uu_1(\omega), \dots , \uu_r(\omega)] \in \CC^{d \times r}
\]
is a sub-unitary matrix of corresponding right eigenvectors, not unique even if all eigenvalues are distinct. Then, still, we have
\begin{equation}\label{eq:spectr_decomp_f_short}
\Lam_r(\omega) = \tilde{\UU}^*(\omega) \f(\omega) \tilde{\UU}(\omega) .
\end{equation}

The matrix function $\Lam_r(\omega)$ is a self-adjoint, positive definite function, and 
\[
\tr(\Lam_r(\omega)) = \tr(\f(\omega)), 
\]
where $\f(\omega)$ is the density function of a finite spectral measure. This shows that the integral of $\tr(\Lam_r(\omega))$ over $[-\pi, \pi]$ is finite. Thus $\Lam_r(\omega)$ can be considered the spectral density function of a full rank regular process. So it can be factored, in fact, we may take a miniphase $H^2$ spectral factor $\DD_r(\omega)$ of it:
\begin{equation}\label{eq:Lamr_factor1}
\Lam_r(\omega) = \frac{1}{2 \pi} \DD_r(\omega) \DD_r(\omega),
\end{equation}
where $\DD_r(\omega)$ is a diagonal matrix.

Then a simple way to factorize $\f$ is to choose
\begin{equation}\label{eq:trivial_spectral_factor}
\phiv(\omega) = \tilde{\UU}(\omega) \DD_r(\omega) \A(\omega) = \tilde{\UU}(\omega) \A(\omega) \DD_r(\omega) = \tilde{\UU}_{\A}(\omega) \DD_r(\omega),
\end{equation}
where $\A(\omega) = \diag[a_1(\omega), \dots , a_r(\omega)]$, each $a_k(\omega)$ being a measurable function on $[-\pi, \pi]$ such that $|a_k(\omega)| = 1$ for any $\omega$, but otherwise arbitrary and $\tilde{\UU}_{\A}(\omega)$ still denotes a sub-unitary matrix of eigenvectors of $\f$ in the same order as the one of the eigenvalues.

To my best knowledge it is not known if for any regular time series $\{\X_t\}$ there exists such a matrix valued function $\A(\omega)$ so that $\phiv(\omega)$ defined by \eqref{eq:trivial_spectral_factor} has a Fourier series with only non-negative powers of $e^{-i \omega}$. Equivalently, does there exist an $H^2$ analytic matrix function $\Phiv(z)$  whose boundary value is the above spectral factor $\phiv(\omega)$ with some $\A(\omega)$, according to the formulas \eqref{eq:f_factor}--\eqref{eq:Phiv}?

\begin{thm}\label{th:generic_reg_nec} See \cite[Theorem 2.1]{szabados2022regular} and \cite[Theorem 2.1]{szabados2022corrigenda}.

\begin{itemize}
\item[(a)] Assume that a $d$-dimensional stationary time series $\{\X_t\}$ is regular of rank $r$, $1 \le r \le d$. Then for $\Lam_r(\omega)$ defined by \eqref{eq:Lam_r} one has $\log \det \Lam_r \in L^1 = L^1([-\pi, \pi], \mathcal{B}, d\omega)$, equivalently,
\begin{equation}\label{eq:log_lambda}
    \int_{-\pi}^{\pi} \log \lambda_r(\omega) \, d\omega > -\infty .
\end{equation}

\item[(b)] If moreover one assumes that the regular time series $\{\X_t\}$ is such that has an $H^2$ spectral factor of the form \eqref{eq:trivial_spectral_factor}, then the following statement holds as well:

The sub-unitary matrix function $\tilde{\UU}(\omega)$ appearing in the spectral decomposition of $\f(\omega)$ in \eqref{eq:spectr_decomp_f_short1} can be chosen so that it belongs to the Hardy space $H^{\infty} \subset H^2$, thus

\begin{equation}\label{eq:UU_Fourier_series}
  \tilde{\UU}(\omega) = \sum_{j=0}^{\infty} \psiv(j) e^{-i j \omega} , \quad \psiv(j) \in \CC^{d \times r}, \quad \sum_{j=0}^{\infty} \|\psiv(j)\|^2 < \infty .
\end{equation}
In this case one may call $\tilde{\UU}(\omega)$ an inner matrix function.
\end{itemize}
\end{thm}

The next theorem gives a sufficient condition for the regularity of a generic weakly stationary time series; compare with the statements of Theorem \ref{th:generic_reg_nec} above. Observe that assumptions (1) and (2) in the next theorem are necessary conditions of regularity as well. Only assumption (3) is not known to be necessary. We think that these assumptions are simpler to check in practice then the ones of Rozanov's theorem cited above. By formula \eqref{eq:trivial_spectral_factor}, checking assumption (3) means that for each  eigenvectors $\uu_k(\omega)$ of $\f(\omega)$ we are searching for a complex function multiplier $a_k(\omega)$ of unit absolute value that gives an $H^{\infty}$ function result.

\begin{thm}\label{th:generic_reg_suf} See \cite[Theorem 2.1]{szabados2022regular} and \cite[Theorem 2.2]{szabados2022corrigenda}.

Let $\{\X_t\}$ be a $d$-dimensional time series. It is regular of rank $r \le d$ if the following three conditions hold.
\begin{itemize}

\item[(1)]
It has an absolutely continuous spectral measure matrix $d\F$ with density matrix $\f(\omega)$ which has rank $r$ for a.e.\ $\omega \in[-\pi, \pi]$.

\item[(2)]
For $\Lam_r(\omega)$ defined by \eqref{eq:Lam_r} one has $\log \det \Lam_r \in L^1 = L^1([-\pi, \pi], \mathcal{B}, d\omega)$, equivalently, \eqref{eq:log_lambda} holds.

\item[(3)]
The sub-unitary matrix function $\tilde{\UU}(\omega)$ appearing in the spectral decomposition of $\f(\omega)$ in \eqref{eq:spectr_decomp_f_short1} can be chosen so that it belongs to the Hardy space $H^{\infty} \subset H^2$, thus \eqref{eq:UU_Fourier_series} holds.

\end{itemize}
\end{thm}

Next we discuss another sufficient condition of regularity of a general stationary time series $\{\X_t\}$. An advantage of this sufficient condition is that it gives the Wold representation of $\{\X_t\}$ as well. Let us begin with the sufficient and necessary condition of regularity given by the factorization \eqref{eq:f_factor}, \eqref{eq:bj_Fourier} of the spectral density $\f$, where the $d \times r$ spectral factor $\phiv$ is in $H^2$.

Using Singular Value Decomposition (SVD), we can write that
\[
\phiv(\omega) = \VV(\omega) \SSS(\omega) \UU^*(\omega),
\]
where $\VV(\omega)$ is a $d \times r$ sub-unitary matrix, $\UU(\omega)$ is an $r \times r$ unitary matrix, $\SSS(\omega) = \diag[s_1, s_2, \dots , s_r]$ is an $r \times r$ diagonal matrix of positive singular values $s_1 \ge s_2 \ge \cdots \ge s_r$, for a.e.\ $\omega \in [-\pi, \pi]$. Clearly, $s_j = \sqrt{\lambda_j}$, for $j=1, \dots, r$.

The (generalized) inverse of $\phiv(\omega)$ is not unique when $d > r$. Let $\psiv(\omega)$ be the Moore--Penrose inverse of $\phiv(\omega)$:
\begin{equation}\label{eq:Moore_Penrose}
\psiv(\omega) := \UU(\omega) \SSS^{-1}(\omega) \VV^*(\omega), \quad \psiv(\omega) \phiv(\omega) = \I_r, \quad \text{a.e.} \quad \omega \in [-\pi, \pi] .
\end{equation}

We also need the spectral (Cram\'er) representation of the stationary time series
\[
\X_t = \int_{-\pi}^{\pi} e^{i t \omega} d \Z_{\omega}, \quad t \in \ZZ,
\]
where $\{\Z_{\omega}\}$ is a stochastic process with orthogonal increments, obtained by the isometry between the Hilbert spaces $L^2([-\pi, \pi], \mathcal{B}, \tr(d\F))$ and $H(\X) \subset L^2(\Omega, \mathcal{F}, \PPP)$.

\begin{thm}\label{th:H2H2}
Assume that the spectral measure of a $d$-dimensional weakly stationary time series $\{\X_t\}$ is absolutely continuous with density $\f$ which has constant rank $r$, $1 \le r \le d$. Moreover, assume that there is a finite constant $M$ such that $\|\f(\omega)\| \le M$ for all $\omega \in [-\pi, \pi]$, and $\f$ has a factorization $f = \frac{1}{2 \pi} \phiv \phiv^*$, where $\phiv \in H^2$ and its Moore--Penrose inverse $\psiv \in H^2$ as well.

Then the time series is regular and its fundamental white noise process can be obtained as
\begin{align}\label{eq:fund_white_noise1}
  \xiv_t &= \int_{-\pi}^{\pi} e^{i t \omega} \psiv(\omega) d\Z_{\omega}  \\
    &= \sum_{k=0}^{\infty} \gammav(k) \X_{t-k} \qquad (t \in \ZZ), \label{eq:fund_white_noise2}
\end{align}
where
\begin{equation}\label{eq:psiv_Fourier}
\psiv(\omega) = \sum_{k=0}^{\infty} \gammav(k) e^{-i k \omega}, \quad \gammav(k) = \frac{1}{2\pi} \int_{-\pi}^{\pi} e^{i k \omega} \psiv(\omega) d\omega
\end{equation}
is the Fourier series of $\psiv$, convergent in $L^2$ sense.

The sequence of coefficients $\{\bv(j)\}$ of the Wold representation is given by the $L^2$ convergent Fourier series
\begin{equation}\label{eq:phiv_Fourier}
\phiv(\omega) = \sum_{j=0}^{\infty} \bv(j) e^{-i j \omega}, \quad \bv(j) = \frac{1}{2\pi} \int_{-\pi}^{\pi} e^{i j \omega} \phiv(\omega) d\omega.
\end{equation}
\end{thm}
\begin{proof}
The regularity of $\{\X_t\}$ obviously follows from the assumptions by \eqref{eq:f_factor} and \eqref{eq:bj_Fourier}.

First let us verify that the stochastic integral \eqref{eq:fund_white_noise1} is correct, that is, the components of $\psiv$ belong to $L^2([-\pi, \pi], \mathcal{B}, \tr(d\F))$:
\[
\int_{-\pi}^{\pi} \psiv(\omega) \f(\omega) \psiv^*(\omega) d\omega = \frac{1}{2\pi} \int_{-\pi}^{\pi} \psiv(\omega) \phiv(\omega) \phiv^*(\omega) \psiv^*(\omega) d\omega = \I_r .
\]
This also justifies that
\[
\xiv_t = \int_{-\pi}^{\pi} e^{i t \omega} \psiv(\omega) d\Z_{\omega} = \int_{-\pi}^{\pi} e^{i t \omega} \sum_{k=0}^{\infty} \gammav(k) e^{-i k \omega} d\Z_{\omega}
= \sum_{k=0}^{\infty} \gammav(k) \X_{t-k} .
\]

Second, let us check that the sequence defined by \eqref{eq:fund_white_noise1} is orthonormal, using the isometry mentioned above:
\begin{align*}
\E(\xiv_n \xiv^*_m) &= \int_{-\pi}^{\pi} e^{i n \omega} \psiv(\omega) \f(\omega) e^{-i m \omega} \psiv^*(\omega) d \omega \\
&= \frac{1}{2 \pi} \int_{-\pi}^{\pi} e^{i (n - m) \omega} \psiv(\omega) \phiv(\omega) \phiv^*(\omega) \psiv^*(\omega) d \omega = \delta_{n, m} \I_r .
\end{align*}

Third, let us show that $\xiv_n$ is orthogonal to the past $H^-_{n-k}(\X)$ for any $k > 0$:
\begin{align*}
\E(\X_{n-k} \xiv^*_n) &= \int_{-\pi}^{\pi} e^{i (n-k) \omega} \f(\omega) e^{-i n \omega} \psiv^*(\omega)  d \omega \\
&= \frac{1}{2 \pi} \int_{-\pi}^{\pi} e^{-i k \omega} \phiv(\omega) \phiv^*(\omega) \psiv^*(\omega) d \omega = \frac{1}{2 \pi} \int_{-\pi}^{\pi} e^{-i k \omega} \phiv(\omega) d \omega = \bm{0}_{d \times r}
\end{align*}
for any $k > 0$, since $\phiv \in H^2$, so its Fourier coefficients with negative indices are zero.

Fourth, let us see that $\xiv_n \in H^-_n(\X)$ for $n \in \ZZ$. Because of stationarity, it is enough to show that $\xiv_0 \in H^-_0(\X)$. Since $H^-_0(\X)$ is the closure in $L^2(\Omega, \mathcal{F}, \PPP)$ of the components of all finite linear combinations of the form $\sum_{k=0}^N \gammav_k \X_{-k}$, by the isometry it is equivalent to the fact that $\psiv$ belongs to the closure in $L^2([-\pi, \pi], \mathcal{B}, \tr(d\F))$ of all finite linear combinations of the form $\sum_{k=0}^N \gammav_k e^{-i k \omega}$.

We assumed that $\psiv \in H^2$, which means that $\psiv$ has a one-sided Fourier series \eqref{eq:psiv_Fourier} which converges in $L^2([-\pi, \pi], \mathcal{B}, d\omega)$, where $d\omega$ denotes Lebesgue measure. Then using the assumed boundedness of $\|\f\|$, we obtain that \begin{align*}
&\int_{-\pi}^{\pi} \left\| \left(\sum_{k=0}^N \gammav_k e^{-i k \omega} - \psiv(\omega) \right) \f(\omega) \left(\sum_{k=0}^N \gammav^*_k e^{i k \omega} - \psiv^*(\omega) \right) \right\| d\omega \\
\le & \int_{-\pi}^{\pi} \left\| \sum_{k=0}^N \gammav_k e^{-i k \omega} - \psiv(\omega) \right\| \| \f(\omega)\| \left\|\sum_{k=0}^N \gammav^*_k e^{i k \omega} - \psiv^*(\omega)  \right\| d\omega \\
\le & M \int_{-\pi}^{\pi} \left\| \sum_{k=0}^N \gammav_k e^{-i k \omega} - \psiv(\omega) \right\|^2 d\omega ,
\end{align*}
which tends to $0$ as $N \to \infty$. This shows that $\xiv_0 \in H^-_0(\X)$.

Fifth, by \eqref{eq:Moore_Penrose}, we see that
\begin{equation}\label{eq:phiv_psiv}
(\phiv \psiv - \I_d) \f (\psiv^* \phiv^* - \I_d) = (\phiv \psiv - \I_d) \frac{1}{2 \pi} \phiv \phiv^* (\psiv^* \phiv^* - \I_d) = \bm{0}_{d \times d},
\end{equation}
a.e.\ in $[-\pi, \pi]$. Consequently, the difference
\[
\Delta_t := \int_{-\pi}^{\pi} e^{i t \omega} \phiv(\omega) \psiv(\omega) d\Z_{\omega} - \int_{-\pi}^{\pi} e^{i t \omega} d\Z_{\omega}
\]
is orthogonal to itself in $H(\X)$, so it is a zero vector. Then by \eqref{eq:fund_white_noise1} and \eqref{eq:phiv_Fourier},
\begin{align}
\X_t &= \int_{-\pi}^{\pi} e^{i t \omega} d\Z_{\omega} = \int_{-\pi}^{\pi} e^{i t \omega} \phiv(\omega) \psiv(\omega) d\Z_{\omega} \nonumber \\
&= \int_{-\pi}^{\pi} e^{i t \omega} \sum_{j=0}^{\infty} \bv(j) e^{-i j \omega}  \psiv(\omega) d\Z_{\omega} = \sum_{j=0}^{\infty} \bv(j) \xiv_{t-j} . \label{eq:Wold_repr}
\end{align}
Equation \eqref{eq:phiv_psiv} shows that each entry of $\phiv \psiv$ belongs to $L^2([-\pi, \pi], \mathcal{B}, \tr(d\F))$, so the isometry between this space and $H(\X)$ justifies \eqref{eq:Wold_repr}.

Finally, the previous steps show that the innovation spaces of the sequences $\{\X_t\}$ and $\{\xiv_t\}$ are the same for any time $n \in \ZZ$, so the pasts $H^-_n(\X)$ and $H^-_n(\xiv)$ agree as well for any $n \in \ZZ$. Thus \eqref{eq:Wold_repr} gives the Wold representation of $\{\X_t\}$.
\end{proof}

\section{Smooth eigenvalues of the spectral density}\label{se:smooth}

In the one-dimensional case there is a well-known sufficient condition of  regularity, which at the same time gives a formula for a $H^2$ spectral factor and also for the white noise sequence and the coefficients in the Wold decomposition \eqref{eq:causal_MA}. This is the assumption that the process has a continuously differentiable spectral density $f(\omega) > 0$ for any $\omega \in [-\pi, \pi]$, see e.g.\ \cite[p. 76]{lamperti1977stochastic} or \cite[Subsection 2.8.2]{bolla2021multidimensional}.

This sufficient condition can be partially generalized to the multidimensional case. When a regular $d$-dimensional time series $\{\X_t\}$ has an $H^2$ spectral factor of the form \eqref{eq:trivial_spectral_factor}, equivalently, has a sub-unitary matrix function $\tilde{\UU}(\omega)$ appearing in the spectral decomposition of $\f(\omega)$ in \eqref{eq:spectr_decomp_f_short1} that can be chosen so that it belongs to the Hardy space $H^{\infty} \subset H^2$, then the smoothness of the nonzero eigenvalues of the spectral density $\f$ gives a formula for an $H^2$ spectral factor.

\begin{thm}\label{th:smooth}
Let $\{\X_t\}$ be a $d$-dimensional time series. It is regular of rank $r \le d$ if the following three conditions hold.
\begin{itemize}

\item[(1)]
It has an absolutely continuous spectral measure matrix $d\F$ with density matrix $\f(\omega)$ which has rank $r$ for a.e.\ $\omega \in[-\pi, \pi]$.

\item[(2)]
Each nonzero eigenvalue $\lambda_j(\omega)$ $(j=1, \dots, r)$ of $\f(\omega)$ is a continuously differentiable positive function on $[-\pi, \pi]$.

\item[(3)]
The sub-unitary matrix function $\tilde{\UU}(\omega)$ appearing in the spectral decomposition of $\f(\omega)$ in \eqref{eq:spectr_decomp_f_short1} can be chosen so that it belongs to the Hardy space $H^{\infty} \subset H^2$, thus \eqref{eq:UU_Fourier_series} holds.

\end{itemize}

Moreover, $\{\X_t\}$ satisfies the conditions of Theorem \ref{th:H2H2} too, so formulas \eqref{eq:fund_white_noise1} -- \eqref{eq:phiv_Fourier} give the Wold representation of $\{\X_t\}$.
\end{thm}
\begin{proof}
Condition (2) implies that each $\log \lambda_j(\omega)$ $(j=1, \dots, r)$ is also a continuously differentiable function on $[-\pi, \pi]$ and so it can be expanded into a uniformly convergent Fourier series
\begin{equation}\label{eq:log_lambda_Fourier}
\log \lambda_j(\omega) = \sum_{n=-\infty}^{\infty} \beta_{j,n} e^{i n \omega}, \quad \beta_{j,-n} = \bar{\beta}_{j,n} .
\end{equation}
Write it as
\begin{align*}
\log \lambda_j(\omega) &= Q_j(\omega) + \bar{Q}_j(\omega) \\
&:= \left(\frac12 \beta_{j, 0} + \sum_{n=1}^{\infty} \bar{\beta}_{j,n} e^{-i n \omega}\right) + \left(\frac12 \beta_{j, 0} + \sum_{n=1}^{\infty} \beta_{j,n} e^{i n \omega}\right) .
\end{align*}

Define
\begin{equation}\label{eq:gammaj}
\gamma_j(\omega) := \exp(Q_j(\omega)) = \sum_{k=0}^{\infty} \frac{1}{k!} \left(\frac12 \beta_{j, 0} + \sum_{n=1}^{\infty} \bar{\beta}_{j,n} e^{-i n \omega}\right)^k.
\end{equation}
Then
\begin{equation}\label{eq:lambdaj}
\lambda_j(\omega) = \gamma_j(\omega) \bar{\gamma}_j(\omega) = \exp(Q_j(\omega)) \exp(\bar{Q}_j(\omega)) , \quad j=1, \dots, r .
\end{equation}
Observe that each $Q_j(\omega)$ and consequently each $\gamma_j(\omega)$ $(j=1, \dots, r)$ is a continuous function on $[-\pi, \pi]$, so is in $L^2(T)$. Moreover, each $Q_j(\omega)$ and consequently each $\gamma_j(\omega)$ has only positive powers of $e^{-i \omega}$ in its Fourier series. So each $\gamma_j$ belongs to the Hardy space $H^2$.

Substitute \eqref{eq:lambdaj} into \eqref{eq:spectr_decomp_f_short1}:
\begin{align*}
\f(\omega) = \tilde{\UU}(\omega) \; \diag[\gamma_1(\omega) \cdots \gamma_r(\omega)] \; \diag[\bar{\gamma}_1(\omega) \cdots \bar{\gamma}_r(\omega)] \; \tilde{\UU}^*(\omega) .
\end{align*}
Thus we can take a spectral factor
\begin{equation}\label{eq:smooth_spect_fact}
\phiv(\omega) := \sqrt{2 \pi} \tilde{\UU}(\omega) \Gammav(\omega), \quad \Gammav(\omega):= \diag[\gamma_1(\omega) \cdots \gamma_r(\omega)] .
\end{equation}
Since each $\gamma_j(\omega) \in H^2$ and by condition (3) each entry of $\tilde{U}(\omega)$ is in $H^{\infty}$, each entry of $\phiv(\omega)$ is in $H^2$. It means that $\phiv$ is an $H^2$ spectral factor as in \eqref{eq:f_factor} and \eqref{eq:bj_Fourier}, consequently $\{\X_t\}$ is regular.

Take the Moore--Penrose inverse of $\phiv$:
\[
\psiv(\omega) := \phiv^+(\omega) = (2 \pi)^{-\frac12} \Gammav^{-1}(\omega) \tilde{\UU}^*(\omega), \quad \Gammav^{-1}(\omega) := \diag[\gamma^{-1}_1(\omega), \dots, \gamma^{-1}_r(\omega)],
\]
where each $\gamma^{-1}_j(\omega) = \exp(-Q_j(\omega)$, so it is also continuous and its Fourier series has only positive powers of $e^{-i \omega}$ too. It implies that $\psiv \in H^2$.

Finally, since each $\lambda_j(\omega)$ is a continuous function on $[-\pi, \pi]$, so bounded, and the components of $\tilde{\UU}(\omega)$ are bounded functions because $\tilde{\UU}(\omega)$ is sub-unitary, it follows that $\|\f\|$ is bounded.
\end{proof}

The $d$-dimensional Kolmogorov--Szeg\H{o} formula \eqref{eq:Kolm_Szego_mult} gives only the determinant of the covariance matrix $\Siga$ of the innovations in the full rank regular time series. Similar is the case when the rank $r$ of the process is less than $d$:
\[
\det \Siga_r = (2\pi)^d \exp \int_{-\pi}^{\pi} \log \det \Lada_r(\omega) \frac{d \omega}{2\pi},
\]
where $\Lada_r$ is the diagonal matrix of the $r$ nonzero eigenvalues of $\f$ and $\Siga_r$ is the covariance matrix of the innovation of an $r$-dimensional subprocess of rank $r$ of the original time series, see \cite[Corollary 4.5]{bolla2021multidimensional} or \cite[Corollary 2.5]{szabados2022regular}. 

Fortunately, under the conditions of Theorem \ref{th:smooth}, one can obtain the covariance matrix $\Siga$ itself by a similar formula, as the next theorem shows.
\begin{thm}\label{th:smooth_K_Sz}
Assume that a weakly stationary $d$-dimensional time series satisfies the conditions of Theorem \ref{th:smooth}. Then the covariance matrix $\Siga$ of the innovations of the process can be obtained as
\[
\Siga = 2 \pi \psiv(0) \; \diag \left[\exp \int_{-\pi}^{\pi} \log \lambda_j(\omega) \frac{d \omega}{2 \pi} \right]_{j=1, \dots, r} \; \psiv^*(0) ,
\]  
where $\lambda_j$, $j=1, \dots, r$, are the nonzero eigenvalues of the spectral density matrix $\f$ of the process, $\tilde{\UU}(\omega)$ is the $d \times r$ matrix of corresponding orthonormal eigenvectors, and
\begin{equation}\label{eq:psiv0}
\psiv(0) = \frac{1}{2 \pi} \int_{-\pi}^{\pi} \tilde{\UU}(\omega) d\omega .
\end{equation}
\end{thm}
\begin{proof}
The error of the best 1-step linear prediction by \eqref{eq:pred_h_step}, and the same time, the innovation is
\[
\X_1 - \hat{\X}_1 = \bv(0) \xiv_1 ,
\]
using the Wold decomposition of $\{\X_t\}$. Thus the covariance of the innovation is
\[
\Siga = \E((\X_1 - \hat{\X}_1) (\X_1 - \hat{\X}_1)^*) = \bv(0) \bv^*(0) .
\]

With the analytic function $\Phiv(z)$ corresponding to the Wold decomposition by \eqref{eq:Phiv}, $\bv(0) = \Phiv(0)$. Taking the Fourier series \eqref{eq:UU_Fourier_series}, let
\[
\hat{\UU}(z) := \sum_{j=0}^{\infty} \psiv(j) z^j, \quad |z| \le 1.
\]
Also, denote by \eqref{eq:gammaj}
\[
\hat{\Gammav}(z) := \diag \left[\sum_{k=0}^{\infty} \frac{1}{k!} \left(\frac12 \beta_{j, 0} + \sum_{n=1}^{\infty} \bar{\beta}_{j,n} z^n \right)^k \right]_{j=1, \dots, n} , \quad |z| \le 1.
\]
Now using \eqref{eq:smooth_spect_fact}, it follows that 
\[
\Phiv(z) = \sqrt{2 \pi} \hat{\UU}(z) \hat{\Gammav}(z), \quad |z| \le 1 ,
\]
and
\[
\Phiv(0) = \sqrt{2 \pi} \hat{\UU}(0) \hat{\Gammav}(0) = \psiv(0) \; \diag \left[ \exp(\beta_{j,0}/2) \right]_{j=1, \dots, r} .
\]

Combining the previous results, 
\[
\Siga = 2 \pi \psiv(0) \; \diag \left[ \exp(\beta_{j,0}) \right]_{j=1, \dots, r} \psiv^*(0) ,
\]
where $\psiv(0)$ is given by \eqref{eq:psiv0} and by \eqref{eq:log_lambda_Fourier}, 
\[
\beta_{j, 0} = \frac{1}{2 \pi} \int_{-\pi}^{\pi} \log \lambda_j(\omega) d\omega, \quad j=1, \dots, r.
\]
This completes the proof of the theorem.
\end{proof}


\end{document}